\input ./style/arxiv-ba.cfg
\documentclass[ba,linksfromyear,preprint]{imsart}
\makeatletter
   \@ifpackageloaded{natbib}{}{\usepackage{natbib}}
\makeatother
\usepackage{bm}
\usepackage{multirow}

\pubyear{2015}
\volume{10}
\issue{4}
\firstpage{909}
\lastpage{936}
\doi{10.1214/14-BA929}

\startlocaldefs
\theoremstyle{plain}
\newtheorem{thm}{\protect\theoremname}
\theoremstyle{remark}
\newtheorem*{rem*}{\protect\remarkname}
\theoremstyle{plain}
\newtheorem{lem}[thm]{\protect\lemmaname}
\theoremstyle{remark}
\newtheorem{rem}{\protect\remarkname}
\providecommand{\lemmaname}{Lemma}
\providecommand{\remarkname}{Remark}
\providecommand{\theoremname}{Theorem}
\def\bysame{---}
\newcommand{\enquote}[1]{``#1''}
\endlocaldefs

\begin{document}

\begin{frontmatter}
\title{Bayesian Variable Selection and Estimation for~Group Lasso}
\runtitle{Bayesian Variable Selection and Estimation for Group Lasso}

\begin{aug}
\author[a]{\fnms{Xiaofan}
\snm{Xu}\corref{}\ead[label=e1]{xiaofanxufl@gmail.com}} and
\author[b]{\fnms{Malay} \snm{Ghosh}\ead[label=e2]{ghoshm@stat.ufl.edu}}

\runauthor{X.  Xu and M. Ghosh}
\address[a]{Department of Statistics, University of Florida, \printead{e1}}
\address[b]{Department of Statistics, University of Florida, \printead{e2}}

\end{aug}

\begin{abstract}
The paper revisits the Bayesian group lasso and uses spike and slab priors
for group variable selection. In the process, the connection of our
model with penalized regression is demonstrated, and the role of posterior
median for thresholding is pointed out. We show that the posterior
median estimator has the oracle property for group variable selection
and estimation under orthogonal designs, while the group lasso has
suboptimal asymptotic estimation rate when variable selection consistency
is achieved. Next we consider bi-level selection problem and propose
the Bayesian sparse group selection again with spike and slab priors
to select variables both at the group level and also within a group.
We demonstrate via simulation that the posterior median estimator
of our spike and slab models has excellent performance for both variable
selection and estimation.
\end{abstract}



\end{frontmatter}


\section{Introduction}

Group structures of predictors arise naturally in many statistical
applications:
\begin{itemize}
\item In a regression model, a multi-level categorical predictor is usually
represented by a group of dummy variables.
\item In an additive model, a continuous predictor may be represented by
a group of basis functions to incorporate nonlinear relationship.
\item Grouping structure of variables may be introduced into a model to
make use of some domain specific prior knowledge. Genes in the same
biological pathway, for example, form a natural group.
\end{itemize}
For a thorough review of the application of group variable selection
methods in statistical problems, one may refer to \citet{huang_selective_2012},
in which semiparametric regression models, varying coefficients models,
seemingly unrelated regressions and analysis of genomic data are discussed.

It is usually desirable to use the prior information on the grouping
structure to select variables group-wise. Depending on the application,
selecting individual variables in a group may or may not be relevant.
We will discuss variable selection methods which only conduct variable
selection at the group level, as well as bi-level selection methods
that select variables both at the group level and within group level.

Specifically, we consider a linear regression problem with $G$ factors
(groups):
\begin{align}
\bm{Y}_{n\times1}=\sum_{g=1}^{G}\bm{X}_{g}\bm{\beta}_{g}+\bm{\epsilon} & ,\label{group_linear_regression}
\end{align}
where $\bm{\epsilon}_{n\times1}\sim\bm{N}_{n}(\bm{0},\sigma^{2}\bm{I}_{n})$,
$\bm{\beta}_{g}$ is a coefficients vector of length $m_{g}$, and
$\bm{X}_{g}$ is an $n\times m_{g}$ covariate matrix corresponding
to the factor $\bm{\beta}_{g}$, $g=1,2,\dots,G$. Let $p$ be the
total number of predictors, so $p=\sum_{g=1}^{G}m_{g}$. In the following
article, we will use factor and group interchangeably to denote a
group of predictors that are formed naturally.

Penalized regression methods have been very popular for the power
to select relevant variables and estimate regression coefficients
simultaneously. Among them the lasso \citep{tibshirani_regression_1996},
which puts an upper bound on the $L_{1}$-norm of the regression coefficients,
draws much attention for its ability to both select and estimate.
A~distinctive feature of the lasso is that it can produce exact 0
estimates, resulting in automatic model selection with suitably chosen
penalty parameter. Least Angle Regression (LARS) makes the lasso even
more attractive because the full lasso solution path can be computed
with the cost of only one least squares estimation by a modified LARS
algorithm \citep{efron_least_2004}.

With multi-factor analysis of variance problems in mind, \citet{yuan_model_2006}
proposed the group lasso which generalizes the lasso in order to select
grouped variables (factors) for accurate prediction in regression.
The group lasso estimator is obtained by solving
\begin{align}
\min_{\bm{\beta}}\left\Vert \bm{Y}-\sum_{g=1}^{G}\bm{X}_{g}\bm{\beta}_{g}\right\Vert _{2}^{2}+\lambda\sum_{g=1}^{G}\|\bm{\beta}_{g}\|_{2} & .
\end{align}
We note that the lasso is a special case of the group lasso when all
the groups have size~1, i.e., $m_{1}=m_{2}=\cdots=m_{G}=1$.

One major issue with the lasso-type estimates is that it is difficult
to give satisfactory standard errors since the limit distribution
of the lasso estimator is very complicated \citep{knight_asymptotics_2000,chatterjee_bootstrapping_2011}.
But the Bayesian formulation of the lasso can produce reliable standard
errors without any extra efforts. \citet{tibshirani_regression_1996}
suggested that the lasso estimator is equivalent to the posterior
mode with independent double exponential prior for each regression
coefficient. Motivated by the fact that the double exponential distribution
can be expressed as a scale mixture of normal distributions, \citet{park_bayesian_2008}
developed a fully Bayesian hierarchical model and an efficient Gibbs
sampler for the lasso problem. \citet{kyung_penalized_2010} later
extended this model and proposed a fairly general fully Bayesian formulation
which could accommodate various lasso variations, including the group
lasso, the fused lasso \citep{tibshirani_sparsity_2004} and the elastic
net \citep{zou_regularization_2005} (see also \citealt{raman_bayesian_2009}).

Zero inflated mixture priors, an important subclass of spike and slab
priors \citep{mitchell_bayesian_1988}, have been utilized towards
a Bayesian approach for variable selection. \citet{george_approaches_1997}
used zero inflated normal mixture priors in the hierarchical formulation
for variable selection in a linear regression model. To select random
effects, \citet{chen_random_2003} allowed some random effects to effectively
drop out of the model by choosing mixture priors with point mass at
zero for the random effects variances in a linear mixed effects model.
\citet{zhao_credible_2012} developed new multiple intervals for selected
parameters under the Bayesian lasso model with zero inflated mixture
priors.

Point mass mixture priors are also studied by \citet{johnstone_needles_2004}
for estimation of possibly sparse sequences of Gaussian observations,
with an emphasis on utilizing the posterior median, which is proven
to be a soft thresholding estimator like the lasso but with data adaptive
thresholds. Heavy tailed distributions like double exponential for
the continuous part of the mixture are advocated for the purpose of
achieving optimal estimation risk. Posterior concentration of such
priors on sparse sequences is studied by \citet{castillo_needles_2012}.

Following \citet{johnstone_needles_2004}, \citet{yuan_efficient_2005}
combined the power of point mass mixture priors and double exponential
distributions in variable selection and estimation, and showed that
the resulting empirical Bayes estimator is closely related to the
lasso estimator. \citet{lykou_bayesian_2013} proposed a similar mixture
prior and focused on specifying the shrinkage parameter $\lambda$
based on Bayes factors. \citet{zhang_bayesian_2014} generalized this
prior for group variable selection and proposed the hierarchical structured
variable selection (HSVS) method for simultaneous selection of grouped
variables and variables within a group. They also extended the HSVS method
to account for within group serial correlations by using Bayesian
fused lasso technique for within group selection. These authors used
an FDR-based variable selection technique at the group level and posterior
credible intervals for selection of within group variables. The
paper considered an interesting application to molecular inversion
probe studies in breast cancer.

In this paper, instead of taking a traditional Bayesian approach to
group lasso problem \citep{kyung_penalized_2010,raman_bayesian_2009},
we will develop a Bayesian group lasso model with spike and slab priors
(hereafter referred to as BGL-SS) for problems that only require variable
selection at the group level. Our procedure consists of a multivariate
point mass mixture prior similar to \citet{zhang_bayesian_2014} and
produces exact 0 estimates at the group level to facilitate
group variable selection. Marginal posterior median is proven to be
a soft thresholding estimator, and can automatically select variables.
Simulation results suggest that while prediction accuracy is comparable
to the group lasso, median thresholding results in substantial reduction
of false positive rate in comparison to the latter.

Another important problem we focus in this paper is the bi-level selection.
\citet{simon_sparse-group_2012} proposed sparse group lasso
to produce exact 0 coefficients at the group level and also within
a group. The sparse group lasso estimator of $\bm{\beta}$ is given
by
\begin{align}
\min_{\bm{\beta}}\left(\biggl\|\bm{Y}-\sum_{g=1}^{G}\bm{X}_{g}\bm{\beta}_{g}\biggr\|_{2}^{2}+\lambda_{1}\|\bm{\beta}\|_{1}+\lambda_{2}\sum_{g=1}^{G}\|\bm{\beta}_{g}\|_{2}\right) & .\label{eq:sgl}
\end{align}

With the prior of the form
\begin{align}
\pi\left(\bm{\beta}\right) & \propto\exp\left\{ -\lambda_{1}\|\bm{\beta}\|_{1}-\lambda_{2}\sum_{g=1}^{G}\|\bm{\beta}_{g}\|_{2}\right\} ,\label{eq:sprior}
\end{align}
the posterior mode for problem \eqref{group_linear_regression} is
equivalent to the sparse group lasso estimator. We will show that
\eqref{eq:sprior} can also be expressed as a scale mixture of normals
and therefore we can easily provide a full Bayesian implementation
of sparse group lasso (BSGL). Next, to improve the BSGL model, which
undershrinks the coefficients and cannot automatically select variables,
we utilize a hierarchical spike and slab prior structure to select
variables both at the group level and within each group. We will refer
to this as Bayesian sparse group selection with spike and slab priors
(BSGS-SS). We will demonstrate the significant improvement in variable selection
and prediction power via simulation examples.

Although our BSGS-SS method is similar to the HSVS method of \citet{zhang_bayesian_2014},
which also focuses on selection of both group variables and variables
within selected groups, it differs from the latter in the following
sense. To select variables within a group, the HSVS method assumes independent
double exponential priors on the regression coefficients and conducts
selection via posterior credible intervals. They need to decide the
significance level and deal with the complex issue of multiplicity
adjustment. Our priors, with another spike and slab distribution at
the individual level, can automatically select and estimate variables
with posterior median thresholding. So our posterior median estimator
can be a good default estimator and has great variable selection and
prediction performance.

We stress that a key point of this paper is to advocate the use of
posterior median estimator in spike and slab type models as an alternative
sparse estimator to the (sparse) group lasso estimator since the former
can also select and estimate at the same time. Under an orthogonal
design, we will show that they are both soft thresholding estimators,
and the median thresholding estimator is consistent in model selection
and has optimal asymptotic estimation rate, while the group lasso
has to sacrifice estimation rate to achieve selection consistency.
The selected model by median thresholding has far lower false positive
rate than the model chosen by lasso methods in all our simulation
examples. It has even slightly better model selection accuracy than
the model with largest posterior probability, which is often a gold
standard for stochastic model selection \citep{george_approaches_1997,geweke_variable_1994}.
Also the prediction performance of posterior median estimator is better
than the corresponding lasso methods and is marginally better than
that of posterior mean. This is not surprising since the latter is
a Bayesian model averaging estimator and is widely believed to have
optimal prediction performance \citep{clyde_bayesian_1999,hoeting_bayesian_1999,brown_bayes_2002}.

\citet{griffin_structuring_2012} also addressed the variable selection
problem. Their goal was not only to examine whether or not just some
of the regression coefficients are zeros, but also whether there exists
some clustering or grouping of random effects. They met their target
by considering normal-gamma priors. In a later paper, \citet{griffin_priors_2013}
used the same priors, but primarily with the objective of robustifying
as well as combining ridge priors with g-priors \citep{zellner_assessing_1986}.

In Section \ref{sec:BGL-SS}, we assume independent multivariate zero
inflated mixture prior for each factor in our fully Bayesian formulation
of the group lasso (BGL-SS), and derive a Gibbs sampler to compute
the posterior mean and median as our estimators of the coefficients.
We introduce posterior median thresholding in this section and prove
a frequentist oracle property of our procedure for orthogonal designs.
Bi-level selection methods are developed in Section \ref{sec:bi-level-selection}.
In Section \ref{sub:BSGL}, we will introduce a fully Bayesian hierarchical
model for the sparse group lasso and an efficient Gibbs sampler.
We further improve this model in Section \ref{sub:BSGL-SS} with spike and
slab type priors and propose the BSGS-SS model in order to automatically
select variables and improve prediction performance. Simulation results
are given in Section \ref{sec:Simulation} in which our BGL-SS and
BSGS-SS methods show significant improvement in variable selection as compared
to the frequentist group lasso and traditional Bayesian group lasso
methods. We conclude with a brief discussion in Section \ref{sec:discussion}.

\section{Bayesian Group Lasso with Spike and Slab Prior (BGL-SS)\label{sec:BGL-SS}}

\subsection{Model Formulation}

We consider the regression problem with grouped variables in \eqref{group_linear_regression}.
\citet{kyung_penalized_2010} demonstrated that the prior
\begin{align}
\pi\left(\bm{\beta}_{g}\right)\propto\exp\left\{ -\frac{\lambda}{\sigma}\|\bm{\beta}_{g}\|_{2}\right\} ,\label{eq:multi_laplace_prior}
\end{align}
a multivariate generalization of the double exponential prior, can
also be expressed as a scale mixture of normals with Gamma hyperpriors.
Specifically, with
\begin{align}
\bm{\beta}_{g}|\tau_{g}^{2},\sigma^{2}\stackrel{ind}{\sim}\bm{N}_{m_{g}}\left(\bm{0},\tau_{g}^{2}\sigma^{2}\bm{I}_{m_{g}}\right),\tau_{g}^{2}\stackrel{ind}{\sim}\text{Gamma}\left(\frac{m_{g}+1}{2},\frac{\lambda^{2}}{2}\right) & ,
\end{align}
the marginal distribution of $\bm{\beta}_{g}$ is of the form \eqref{eq:multi_laplace_prior}.
This Bayesian formulation encourages shrinkage of coefficients at
the group level and provides comparable prediction performance with
the group lasso. However, this approach, based on estimation of $\bm{\beta}_{g}(g=1,\dots,G)$
by posterior means or medians, does not produce exact 0 estimates.
To introduce sparsity at the group level and facilitate group variable
selection, we assume a multivariate zero inflated mixture prior for
each $\bm{\beta}_{g}$. We propose the following hierarchical Bayesian
group lasso model with an independent spike and slab type prior for
each factor $\bm{\beta}_{g}$:
\begin{gather}
\bm{Y}|\bm{X},\bm{\beta},\sigma^{2}\sim\bm{N}_{n}(\bm{X}\bm{\beta},\sigma^{2}\bm{I}_{n}),\label{eq:model}\\
\bm{\beta}_{g}|\sigma^{2},\tau_{g}^{2}\stackrel{ind}{\sim}(1-\pi_{0})\bm{N}_{m_{g}}(\bm{0},\sigma^{2}\tau_{g}^{2}\bm{I}_{m_{g}})+\pi_{0}\delta_{0}(\bm{\beta}_{g}),\quad g=1,2,\dots,G,\label{eq:glp1}\\
\tau_{g}^{2}\stackrel{ind}{\sim}\text{Gamma}\left(\frac{m_{g}+1}{2},\frac{\lambda^{2}}{2}\right),\quad g=1,2,\dots,G,\label{eq:glp2}\\
\sigma^{2}\sim\text{Inverse Gamma}\left(\alpha,\gamma\right),\quad \sigma^{2}>0,
\end{gather}
where $\delta_{0}(\bm{\beta}_{g})$ denotes a point mass
at $\bm{0}\in\mathbb{R}^{m_{g}}$, $\bm{\beta}_{g}=(\beta_{g1},\dots,\beta_{gm_{g}})^{T}$.
In this paper, a limiting improper prior is used for $\sigma^{2}$,
$\pi(\sigma^{2})=1/\sigma^{2}$.

Fixing $\pi_{0}$ at $\frac{1}{2}$ is a popular choice since it assigns
equal prior probabilities to all submodels and represents no prior
information on the true model. Instead of fixing $\pi_{0}$, we place
a conjugate beta prior on it, $\pi_{0}\sim\text{Beta}(a,b)$. We prefer
$a=b=1$ since it gives a prior mean $\frac{1}{2}$ and  also allows
a prior spread. Under sparsity, for example, in gene selection problems,
one may need $\pi_{0}\equiv\pi_{0n}$ where $\pi_{0n}\rightarrow1$
as $n\rightarrow\infty$.

The value of $\lambda$ should be carefully tuned. A very large value
of $\lambda$ will overshrink the coefficients and thus yields severely
biased estimates; $\lambda\rightarrow0$ will lead to a very diffuse
distribution for the slab part and the null model will always be preferred
no matter what data we have because of the Lindley paradox \citep{lindley_statistical_1957}.
A conjugate gamma prior can be placed on the penalty parameter, $\lambda^{2}\sim\text{Gamma}(r,\delta)$.
Instead, we will take an empirical Bayes approach and estimate $\lambda$
from data using marginal maximum likelihood. Since marginal likelihood
function for $\lambda$ is intractable, a Monte Carlo EM algorithm
\citep{casella_empirical_2001,park_bayesian_2008} can be used to
estimate $\lambda$. The $k$th EM update for $\lambda$ is
\[
\lambda^{\left(k\right)}=\sqrt{\frac{p+G}{\sum_{g=1}^{G}E_{\lambda^{\left(k-1\right)}}\left[\tau_{g}^{2}\mid\bm{Y}\right]}},
\]
in which the posterior expectation of $\tau_{g}^{2}$ will be replaced
by the sample average of $\tau_{g}^{2}$ generated in the Gibbs sampler
based on $\lambda^{(k-1)}$.

It should be noted that \eqref{eq:glp1} is essentially a special
case of the prior used in \citet{zhang_bayesian_2014} which conducts
shrinkage at both the group level and also the individual level by
using independent exponential hyperpriors to induce lasso shrinkage
for individual variables. However, focusing on group level selection
only, BGL-SS instead uses group lasso prior on the slab part and is
tailored for problems that only require group level sparsity.

\subsection{Marginal Prior for \texorpdfstring{$\bm{\beta}_g$}{beta} and Connection
with Penalized Regression}

Integrating out $\tau_{g}^{2}$ in \eqref{eq:glp1} and \eqref{eq:glp2},
the marginal prior for $\bm{\beta}_{g}$ is a mixture of point mass
at $\bm{0}\in\mathbb{R}^{m_{g}}$ and a Multi-Laplace distribution:
\begin{align}
\bm{\beta}_{g}|\sigma^{2}\sim\left(1-\pi_{0}\right)\text{M-Laplace}\left(\bm{0},\frac{\sigma}{\lambda}\right)+\pi_{0}\delta_{0}\left(\bm{\beta}_{g}\right) & ,\label{marginal}
\end{align}
where the density function for an $m_{g}$-dimensional Multi-Laplace
distribution is
\begin{align}
\text{M-Laplace}\left(\bm{x}|0,c^{-1}\right)\propto c^{m_{g}}\exp\left(-c\|x\|_{2}\right) & .
\end{align}

We can observe from \eqref{marginal} that the marginal prior for
$\bm{\beta}_{g}$ has two shrinkage effects: one is the point mass
at $\bm{0}$ which leads to exact $0$ coefficients; the other, same
as the one considered in the Bayesian group lasso \citep{kyung_penalized_2010,raman_bayesian_2009},
results in shrinkage at the group level. Combining these two components
together facilitates variable selection at the group level and shrinks
coefficients in the selected groups at the same time. For the special
case when the dimension of $\bm{\beta}_{g}$ is 1, i.e., $m_{g}=1$,
\eqref{marginal} reduces to a one-dimensional mixture distribution
with a point mass at $0$ and a double exponential distribution. This
has been thoroughly studied by \citet{johnstone_needles_2004} and
\citet{castillo_needles_2012} for estimation of sparse normal means,
and by \citet{yuan_efficient_2005} and \citet{lykou_bayesian_2013}
for Bayesian variable selection. Importantly, it was shown that a
heavy-tailed distribution for the slab part, such as a double-exponential
distribution or a Cauchy-like distribution, is advantageous since
that it results in optimal estimation risk with posterior median estimator
and optimal posterior contraction rate for sparse means. We will generalize
the thresholding result of \citet{johnstone_needles_2004} on the
posterior median to our multivariate spike and slab type prior \eqref{eq:glp1}.

To see the connection between our model and the penalized regression
problem, we reparametrize the regression coefficients: $\bm{\beta}_{g}=\gamma_{g}\bm{b}_{g}$,
where $\gamma_{g}$ is an indicator that only takes value $0$ or
$1$, and $\bm{b}_{g}=(b_{g1},b_{g2},\dots,b_{gm_{g}})^{T}$.
We then place a Multi-Laplace prior on $\bm{b}_{g}$ and a Bernoulli
prior on $\gamma_{g}$,
\begin{align}
\bm{b}_{g}|\sigma & \stackrel{ind}{\sim}\text{M-Laplace}\left(\bm{0},\frac{\sigma}{\lambda}\right),\quad g=1,2,\dots,G,\label{eq:BernoulliGaussian}\\
\gamma_{g} & \stackrel{ind}{\sim}\text{Bernoulli}\left(1-\pi_{0}\right),\quad g=1,2,\dots,G.
\end{align}

Note that with this configuration, the marginal prior distribution
of $\bm{\beta}_{g}$ is still \eqref{marginal} and this model can
only be identified up to $\bm{\beta}_{g}=\gamma_{g}\bm{b}_{g}$. The
negative log-likelihood under the model \eqref{group_linear_regression}
and the above prior is
\[
-\log L\left(\bm{b},\bm{\gamma}|\bm{Y}\right)=\frac{1}{2\sigma^{2}}\|\bm{Y}-\bm{X}\bm{\beta}\|_{2}^{2}+\frac{\lambda}{\sigma}\sum_{g=1}^{G}\|\bm{b}_{g}\|_{2}+\log\left(\frac{1-\pi_{0}}{\pi_{0}}\right)\sum_{g=1}^{G}\gamma_{g}+const.
\]
Thus the posterior mode of the regression model \eqref{group_linear_regression}
under this new parametrization is equivalent to the solution of a
penalized regression problem with an $L_{2}$-penalty on each group
of coefficients and an $L_{0}$-like penalty, penalizing the number
of nonzero groups in the predictors. Solving this penalization regression
problem is extremely hard for problems with a moderate to large number
of groups of covariates because of the combinatorial optimization
problem induced by the $L_{0}$-like norm. We would also like to point
out that for the special case when all the groups have size 1, if
we replace the Laplace prior with Normal prior, it becomes the so-called Bernoulli--Gaussian model or Binary Mask model, and has been
applied to variable selection \citep{kuo_variable_1998} and signal
process problems \citep{zhou_non-parametric_2009,soussen_bernoulli-gaussian_2011}.

\subsection{Posterior Median as an Adaptive Thresholding Estimator}

Regarding the wavelet-based nonparametric problem, \citet{abramovich_wavelet_1998}
demonstrated that the traditional Bayes rule with respect to $L_{2}$-loss function is a shrinkage rule while the posterior median, which
is a Bayes estimator corresponding to $L_{1}$-loss, is a thresholding
estimator with spike and slab priors. \citet{johnstone_needles_2004}
showed that under spike and slab priors for normal means problem,
the posterior median is a random thresholding estimator with a couple
of desirable properties under fairly general conditions. In this section,
we will generalize the thresholding results of \citet{johnstone_needles_2004}
to multivariate spike and slab priors, with \eqref{eq:glp1} as a
special case. First, we focus on only one group:
\begin{equation}
\bm{Z}_{m\times1}\sim f\left(\bm{z}-\bm{\mu}\right),
\end{equation}
\begin{equation}
\bm{\mu}\sim\pi_{0}\delta_{0}\left(\bm{\mu}\right)+\left(1-\pi_{0}\right)\gamma\left(\bm{\mu}\right),
\end{equation}
where $\bm{Z}$ is an $m$-dimensional random variable, and $\gamma(\cdot)$
and $f(\cdot)$ are both density functions for $m$-dimensional
random vectors. $f(\bm{t})$ is maximized at $\bm{t}=\bm{0}$.
Let $\text{Med}(\mu_{i}|\bm{z})$ denote the marginal posterior
median of $\mu_{i}$ given data. We define
\[
c=\frac{\int f\left(-\bm{v}\right)\gamma\left(\bm{v}\right)d\bm{v}}{f\left(\bm{0}\right)}\leq\frac{\int f\left(\bm{0}\right)\gamma\left(\bm{v}\right)d\bm{v}}{f\left(\bm{0}\right)}=1,
\]
Then we have the following theorem:
\begin{thm}\label{thm1}
Suppose $\pi_{0}>\frac{c}{1+c}$,
then there exists a threshold $t(\pi_{0})>0$, such that
when $\|\bm{z}\|_{2}<t$,
\[
\text{Med}\left(\mu_{i}|\bm{z}\right)=0,\text{ for any }1\leq i\leq m.
\]
\end{thm}
\begin{proof} The posterior odds of $\bm{\mu}\neq\bm{0}$
given $\bm{Z}=\bm{0}$ is
\begin{align*}
Odds\left(\bm{\mu}\neq\bm{0}|\bm{Z}=\bm{0}\right) & =\frac{1-\pi_{0}}{\pi_{0}}\frac{\int f\left(\bm{0}-\bm{v}\right)\gamma\left(\bm{v}\right)d\bm{v}}{f\left(\bm{0}\right)}\\
 & =\frac{1-\pi_{0}}{\pi_{0}}c\\
 & <1.
\end{align*}
Note that $Odds(\bm{\mu}\neq\bm{0}|\bm{Z}=\bm{z})$ is
a continuous function of $\bm{z}$. Hence, there exists $t(\pi_{0})>0$,
such that when $\|\bm{z}\|_{2}<t$, $Odds(\bm{\mu}\neq\bm{0}|\bm{Z}=\bm{z})<1$.
Therefore, when $\|\bm{z}\|_{2}<t$, for any $i(1\leq i\leq m)$,
$P(\mu_{i}=0|\bm{Z}=\bm{z})\geq P(\bm{\mu}=\bm{0}|\bm{Z}=\bm{z})>\frac{1}{2}$,
and we conclude that $\text{Med}(\mu_{i}|\bm{z})=0$.
\end{proof}
Suppose now the design matrix $X$ in (\xch{\ref{eq:model}}{2.3}) is block orthogonal, i.e.,
$\bm{X}_{i}^{T}\bm{X}_{j}=\bm{0}$ for $i\neq j$. Then for $1\leq g\leq G$
we have
\[
\hat{\bm{\beta}}_{g}=\left(\bm{X}_{g}^{T}\bm{X}_{g}\right)^{-1}\bm{X}_{g}^{T}\bm{Y}\sim N_{m_{g}}\left(\bm{\beta}_{g},\sigma^{2}\left(\bm{X}_{g}^{T}\bm{X}_{g}\right)^{-1}\right).
\]
By Theorem \xch{\ref{thm1}}{1}, suppose $\pi_{0}>\frac{c}{1+c}$, then there exists
$t(\pi_{0})>0$, such that the marginal posterior median
of $\beta_{gj}$ under the prior (\xch{\ref{eq:glp1}}{2.4}) satisfies
\[
\text{Med}\left(\beta_{gj}|\hat{\bm{\beta}}_{g}\right)=0\quad\text{for any }1\leq j\leq m_{g}
\]
when $\|\hat{\bm{\beta}}_{g}\|_{2}<t$. Thus the marginal posterior
median estimator of the $g$th group of regression coefficients is
zero when the norm of the corresponding block least square estimator
is less than certain threshold.

To illustrate the random thresholding property of posterior median
estimator, we further assume that the design matrix $\bm{X}$ is orthogonal,
i.e., $\bm{X}^{T}\bm{X}=n\bm{I}_{p}$ for the rest of this subsection
and consider the model defined by \eqref{eq:model} and \eqref{eq:glp1}
with fixed $\tau_{g,n}^{2}(1\leq g\leq G)$. Note that
we use the subscript $n$ here to emphasize that $\tau_{g}^{2}$ depends
on $n$ for developing the asymptotic theory. Under this model, the
posterior distribution of $\bm{\beta}_{g}$ conditional on the data
is still a multivariate spike and slab distribution,
\[
\bm{\beta}_{g}|\bm{Y},\bm{X}\sim l_{g,n}\delta_{0}\left(\bm{\beta}_{g}\right)+\left(1-l_{g,n}\right)\bm{N}_{m_{g}}\left(\left(1-B_{g,n}\right)\hat{\bm{\beta}}_{g}^{LS},\frac{\sigma^{2}}{n}\left(1-B_{g,n}\right)\bm{I}\right),
\]
where $\hat{\bm{\beta}}_{g}^{LS}$ is the least squares estimator of
$\bm{\beta}_{g}$, $B_{g,n}=\frac{1}{1+n\tau_{g,n}^{2}}$, and
\[
l_{g,n}=P\left(\bm{\beta}_{g}=\bm{0}|\bm{Y},\bm{X}\right)=\frac{\pi_{0}}{\pi_{0}+\left(1-\pi_{0}\right)\left(1+n\tau_{g,n}^{2}\right)^{-m_{g}/2}\exp\left\{ \frac{\left(1-B_{g,n}\right)}{2\sigma^{2}}n\|\bm{\hat{\beta}}_{g}^{LS}\|_{2}^{2}\right\} }.
\]
Thus the marginal posterior distribution for $\beta_{gj}(1\leq j\leq m_{g})$
conditional on the observed data is also a spike and slab distribution,
\[
\beta_{gj}|\bm{Y},\bm{X}\sim l_{g,n}\delta_{0}\left(\beta_{gj}\right)+\left(1-l_{g,n}\right)N\left(\left(1-B_{g,n}\right)\hat{\beta}_{gj}^{LS},\frac{\sigma^{2}}{n}\left(1-B_{g,n}\right)\right).
\]
The resulting median, a soft thresholding estimator, is given by
\begin{equation}
\hat{\beta}_{gj}^{Med}\stackrel{\triangle}{=}\text{Med}\left(\beta_{gj}|\bm{Y},\bm{X}\right)=\text{sgn}\left(\hat{\beta}_{gj}^{LS}\right)\left(\left(1-B_{g,n}\right)|\hat{\beta}_{gj}^{LS}|-\frac{\sigma}{\sqrt{n}}Q_{g,n}\sqrt{1-B_{g,n}}\right)_{+},\label{eq:median-thresholding}
\end{equation}
where $z_{+}$ denotes the positive part of $z$, and
$Q_{g,n}=\Phi^{-1}(\frac{1}{2(1-\text{min}(\frac{1}{2},l_{g,n}))})$.
This is similar to the group lasso estimator \citep{yuan_model_2006}
which can also be expressed as a soft thresholding estimator under
an orthogonal design:
\[
\hat{\beta}_{gj}^{GL}=\left(1-\frac{\lambda_{n}}{n\|\hat{\bm{\beta}}_{g}^{LS}\|_{2}}\right)_{+}\hat{\beta}_{gj}^{LS}=\text{sgn}\left(\hat{\beta}_{gj}^{LS}\right)\left(|\hat{\beta}_{gj}^{LS}|-\frac{\lambda_{n}}{n}\cdot\frac{|\hat{\beta}_{gj}^{LS}|}{\|\hat{\bm{\beta}}_{g}^{LS}\|_{2}}\right)_{+}.
\]
It should be noted that the $L_{2}$-norm of the shrinkage vector
for the $g$th group is $\lambda_{n}/n$, which is a fixed amount
and does not relate to the relative importance of each factor. It
is expected that such a penalty could be excessive and adversely affect
the estimation efficiency and model selection consistency \citep{wang_note_2008}.
We will demonstrate this point for an orthogonal design.
 \begin{rem}
One interesting observation from \eqref{eq:median-thresholding} is
the interaction of the spike part and the slab part in the posterior
inference. The spike part leads to a soft thresholding estimator that
can select variables automatically and the thresholds depend on $\pi_{0}$,
while the hyperparameter in the slab part, $\tau_{g,n}^{2}$ (or $\lambda$
if the gamma hyperprior is assumed) decides the shrinkage factor $B_{g,n}$.
\end{rem}
Let $\bm{\beta}^{0},\bm{\beta}_{g}^{0},\beta_{gj}^{0}$
denote the true values of $\bm{\beta},\bm{\beta}_{g},\beta_{gj}$,
respectively. Define the index vector of the true model as
$\mathcal{A}=(I(\|\bm{\beta}_{g}\|_{2}\neq0),g=1,2,\dots,G)$,
and the index vector of the model selected by certain thresholding
estimator $\hat{\bm{\beta}}_{g}$ as
$\mathcal{A}_{n}=(I(\|\hat{\bm{\beta}}_{g}\|_{2}\neq0),g=1,2,\dots,\allowbreak G)$.
Model selection consistency is attained if and only if $\lim_{n}P(\mathcal{A}_{n}=\mathcal{A})=1$.
\begin{lem} If $\lambda_{n}/\sqrt{n}\rightarrow\lambda_{0}\geq0$,
then $\lim\sup_{n}P(\mathcal{A}_{n}^{GL}=\mathcal{A})<1$.\end{lem}
\begin{proof} Note that for any $g$ such that $\|\bm{\beta}_{g}\|_{2}=0$,
\[
P\left(\|\hat{\bm{\beta}}_{g}^{GL}\|_{2}=0\right)=P\left(\|\hat{\bm{\beta}}_{g}^{LS}\|_{2}\leq\frac{\lambda_{n}}{n}\right)=P\left(\|\sqrt{n}\hat{\bm{\beta}}_{g}^{LS}\|_{2}\leq\frac{\lambda_{n}}{\sqrt{n}}\right),
\]
where
$\sqrt{n}\hat{\bm{\beta}}_{g}^{LS}\stackrel{d}{\rightarrow}\bm{Z}$,
$\bm{Z}\sim\bm{N}(\bm{0},\bm{I})$,
and
$\lambda_{n}/\sqrt{n}\rightarrow\lambda_{0}\geq0$. Thus by Fatou's
Lemma,
\[
\lim\sup_{n}P\left(\mathcal{A}_{n}^{GL}=\mathcal{A}\right)\leq\lim\sup_{n}P\left(\|\hat{\bm{\beta}}_{g}^{GL}\|_{2}=0\right)\leq P\left(\|\bm{Z}\|_{2}\leq\lambda_{0}\right)<1.\qedhere
\]%
\end{proof}%
We can observe from above lemma that in order for the
group lasso to consistently select variables, we must have $\lambda_{n}/\sqrt{n}\rightarrow\infty$.
But this condition does not give optimal estimation rate, as demonstrated
by the following lemma.
\begin{lem}
If $\lambda_{n}/\sqrt{n}\rightarrow\infty$, then
\[
\frac{n}{\lambda_{n}}\left(\hat{\bm{\beta}}^{GL}-\bm{\beta}^{0}\right)\stackrel{p}{\rightarrow}\bm{C},
\]
where $\bm{C}=(\bm{\beta}_{g}^{0}/\|\bm{\beta}_{g}^{0}\|_{2},g=1,\dots,G)^{T}$,
is a vector of constants depending on the true model.
\end{lem}
\begin{proof}
For any $g(1\leq g\leq G)$,
\begin{align*}
&\frac{n}{\lambda_{n}}\left(\hat{\bm{\beta}}_{g}^{GL}-\bm{\beta}_{g}^{0}\right)\\
&\quad{}=\frac{\sqrt{n}}{\lambda_{n}}\sqrt{n}\left(\hat{\bm{\beta}}_{g}^{LS}-\bm{\beta}_{g}^{0}\right)-\frac{n}{\lambda_{n}}\left(1-\frac{\lambda_{n}}{n\|\hat{\bm{\beta}}_{g}^{LS}\|_{2}}\right)I\left(n\|\hat{\bm{\beta}}_{g}^{LS}\|_{2}<\lambda_{n}\right)-\frac{1}{\|\hat{\bm{\beta}}_{g}^{LS}\|_{2}}\hat{\bm{\beta}}_{g}^{LS}\\
&\quad{}\stackrel{p}{\rightarrow}-\frac{1}{\|\bm{\beta}_{g}^{0}\|_{2}}\bm{\beta}_{g}^{0}
\end{align*}
by noting that $\sqrt{n}(\hat{\bm{\beta}}_{g}^{LS}-\bm{\beta}_{g}^{0})=O_{p}(1)$,
$\frac{\sqrt{n}}{\lambda_{n}}\rightarrow0$,
$I(n\|\hat{\bm{\beta}}_{g}^{LS}\|_{2}<\lambda_{n})\stackrel{p}{\rightarrow}0$
and applying Slutsky's theorem. \end{proof} Thus the convergence
rate of the group lasso estimator is $n/\lambda_{n}$, which is slower
than $\sqrt{n}$. Adaptive group lasso \citep{wang_note_2008,nardi_asymptotic_2008}
was proposed to overcome this limitation. By using different regularization
parameter that depends on the least square estimators for different
factors, the adaptive group lasso enjoys oracle property. We will
show that the median thresholding estimator also has the oracle property
under an orthogonal design. \begin{thm} Assume orthogonal design
matrix, i.e., $\bm{X}^{T}\bm{X}=n\bm{I}_{p}$. Suppose $\sqrt{n}\tau_{g,n}^{2}\rightarrow\infty$
and $\log(\tau_{g,n}^{2})/n\rightarrow0$ as $n\rightarrow\infty$,
for $g=1,\dots,G$, then the median thresholding estimator has oracle
property, that is, variable selection consistency,
\[
\lim_{n\rightarrow\infty}P\left(\mathcal{A}_{n}^{Med}=\mathcal{A}\right)=1
\]
and asymptotic normality,
\[
\sqrt{n}\left(\hat{\bm{\beta}}_{\mathcal{A}}^{Med}-\bm{\beta}_{\mathcal{A}}^{0}\right)\stackrel{d}{\rightarrow}\bm{N}\left(\bm{0},\sigma^{2}\bm{I}\right).
\]
\end{thm} \begin{proof} First we observe that $\lim_{n\rightarrow\infty}\sqrt{n}B_{g,n}=0$
since $\sqrt{n}\tau_{g,n}^{2}\rightarrow\infty$ as $n\rightarrow\infty$,
$g=1,\dots,G$.

For $g$ such that $\|\bm{\beta}_{g}^{0}\|_{2}=0$, since
$\sqrt{n}\hat{\bm{\beta}}_{g}^{LS}=O_{p}(1)$
and $n\tau_{g,n}^{2}\rightarrow\infty$, $l_{g,n}\stackrel{p}{\rightarrow}1$
as $n\rightarrow\infty$. The probability of correctly classifying
this factor is
\begin{align*}
P\left(\|\hat{\bm{\beta}}_{g}^{Med}\|_{2}=0\right) & =P\left(\left(1-B_{g,n}\right)|\hat{\beta}_{gj}^{LS}|\leq\frac{\sigma}{\sqrt{n}}Q_{g,n}\sqrt{1-B_{g,n}},j=1,...,m_{g}\right)\\[4pt]
 & =\prod_{j=1}^{m_{g}}P\left(T_{g,n}^{j}\leq1\right)\\[4pt]
 & \rightarrow1\text{ as }n\rightarrow\infty
\end{align*}
where $T_{g,n}^{j}\stackrel{\triangle}{=}\sigma\sqrt{1-B_{g,n}}\cdot\sqrt{n}\hat{|\beta}_{gj}^{LS}|/Q_{g,n}\stackrel{p}{\rightarrow}0$
for all $1\leq j\leq m_{g}$ by Slutsky's theorem.

For $g$ such that $\|\bm{\beta}_{g}^{0}\|_{2}\neq0$, since $\hat{\bm{\beta}}_{g}^{LS}\stackrel{p}{\rightarrow}\bm{\beta}_{g}^{0}$
and $\log(\tau_{g,n}^{2})/n\rightarrow0$, $l_{g,n}\stackrel{p}{\rightarrow}0$
as $n\rightarrow\infty$. The probability of correctly identifying
this factor is
\begin{align*}
P\left(\|\hat{\bm{\beta}}_{g}^{Med}\|_{2}\neq0\right) & =P\left(\left(1-B_{g,n}\right)|\hat{\beta}_{gj}^{LS}|>\frac{\sigma}{\sqrt{n}}Q_{g,n}\sqrt{1-B_{g,n}},j=1,...,m_{g}\right)\\[4pt]
 & =\prod_{j=1}^{m_{g}}P\left(1/T_{g,n}^{j}<1\right)\\[4pt]
 & \rightarrow1\text{ as }n\rightarrow\infty
\end{align*}
where $1/T_{g,n}^{j}\stackrel{p}{\rightarrow}0$ for all $1\leq j\leq m_{g}$
by Slutsky's theorem. Thus we have proved variable selection consistency.
For asymptotic normality, we only need to show that
$\sqrt{n}(\hat{\beta}_{gj}^{Med}-\hat{\beta}_{gj}^{LS})
\stackrel{p}{\rightarrow}0$,
and then the result follows from the fact that
$\sqrt{n}(\hat{\beta}_{gj}^{LS}-\beta_{gj}^{0})\stackrel{d}{\rightarrow}N
(0,\sigma^{2})$.
Note that $\sqrt{n}B_{g,n}\rightarrow0$, $\hat{\bm{\beta}}_{g}^{LS}\stackrel{p}{\rightarrow}\bm{\beta}_{g}^{0}$,
$l_{g,n}\rightarrow0$ and $\sqrt{n}I(T_{g,n}^{j}\leq1)\stackrel{p}{\rightarrow}0$.
Then
\begin{align*}
&\left|\sqrt{n}\left(\hat{\beta}_{gj}^{Med}-\hat{\beta}_{gj}^{LS}\right)\right| \\
&\quad{} =\left(\sqrt{n}B_{g,n}|\hat{\beta}_{gj}^{LS}|-\sqrt{1-B_{g,n}}Q_{g,n}\right)I\left(T_{g,n}^{j}>1\right)+\sqrt{n}|\hat{\beta}_{gj}^{LS}|I\left(T_{g,n}^{j}\leq1\right)\\
&\quad{} \stackrel{p}{\rightarrow}0%
\end{align*}%
by Slutsky's theorem. Therefore, we conclude
$\sqrt{n}(\hat{\bm{\beta}}_{\mathcal{A}}^{Med}-\bm{\beta}_{\mathcal{A}}^{0})
\stackrel{d}{\rightarrow}\bm{N}
(\bm{0},\sigma^{2}I)$.
\end{proof}

\subsection{Gibbs Sampler\label{sub:BGL-SS-Sampler}}

The full posterior distribution of all the unknown parameters conditional
on data is
\begin{align*}
&{} p(\bm{\beta},\bm{\tau}^{2},\sigma^{2},\pi_{0}|\bm{Y},\bm{X})\\
&\quad{} \propto(\sigma^{2})^{-\frac{n}{2}}\exp\left\{ -\frac{1}{2\sigma^{2}}(\bm{Y}-X\bm{\beta})^{T}(\bm{Y}-X\bm{\beta})\right\} \\
&\quad{} \times\prod_{g=1}^{G}\left[(1-\pi_{0})(2\pi\sigma^{2}\tau_{g}^{2})^{-\frac{m_{g}}{2}}\exp\left\{ -\frac{\bm{\beta}_{g}^{T}\bm{\beta}_{g}}{2\sigma^{2}\tau_{g}^{2}}\right\} I\left[\bm{\beta}_{g}\neq\bm{0}\right]+\pi_{0}\delta_{0}(\bm{\beta}_{g})\right]\\
&\quad{} \times\prod_{g=1}^{G}\left(\lambda^{2}\right)^{\frac{m_{g}+1}{2}}\left(\tau_{g}^{2}\right)^{\frac{m_{g}+1}{2}-1}\exp\left(-\frac{\lambda^{2}}{2}\tau_{g}^{2}\right)\\[3pt]
&\quad{} \times\pi_{0}^{a-1}\left(1-\pi_{0}\right)^{b-1}\\[3pt]
&\quad{} \times\left(\sigma^{2}\right)^{-\alpha-1}\exp\left\{ -\frac{\gamma}{\sigma^{2}}\right\}.
\end{align*}
We utilize an efficient block Gibbs sampler \citep{hobert_geometric_1998}
to simulate from the posterior distribution above. To estimate the
highest posterior probability model, we record the model selected
at each simulation and tabulate them to find the model that appears
most often. Let $\bm{\beta}_{(g)}$ denote the $\bm{\beta}$ vector
without the $g$th group, that is,
\begin{align*}
\bm{\beta}_{(g)}=\left(\bm{\beta}_{1}^{T},\dots,\bm{\beta}_{g-1}^{T},\bm{\beta}_{g+1}^{T},\dots,\bm{\beta}_{G}^{T}\right)^{T} & .
\end{align*}
Let $\bm{X}_{(g)}$ denote the covariate matrix corresponding to $\bm{\beta}_{(g)}$,
that is,
\begin{align*}
\bm{X}_{(g)}=\left(\bm{X}_{1},\dots,\bm{X}_{g-1},\bm{X}_{g+1},\dots,\bm{X}_{G}\right) & ,
\end{align*}
where $\bm{X}_{g}$ is the design matrix corresponding to $\bm{\beta}_{g}$.

The Gibbs Sampler we used to generate from the posterior distribution
is given below
\begin{itemize}
\item Let $\bm{\mu}_{g}=\bm{\Sigma}_{g}\bm{X}_{g}^{T}(\bm{Y}-\bm{X}_{(g)}\bm{\beta}_{(g)}),\bm{\Sigma}_{g}=(\bm{X}_{g}^{T}\bm{X}_{g}+\frac{1}{\tau_{g}^{2}}\bm{I}_{m_{g}})^{-1}$,
then the conditional posterior distribution of $\bm{\beta}_{g}$ is
a spike and slab distribution,
\begin{align*}
\bm{\beta}_{g}|\text{rest} & \sim(1-l_{g})\bm{N}(\bm{\mu}_{g},\sigma^{2}\bm{\Sigma}_{g})+l_{g}\delta_{0}(\bm{\beta}_{g}),\quad g=1,\dots,G,
\end{align*}
where
\begin{align*}
l_{g} & =p(\bm{\beta}_{g}=0|\text{rest})\\
 & =\frac{\pi_{0}}{\pi_{0}+(1-\pi_{0})(\tau_{g}^{2})^{-\frac{m_{g}}{2}}|\bm{\Sigma}_{g}|^{\frac{1}{2}}\exp\left\{ \frac{1}{2\sigma^{2}}\|\bm{\Sigma}_{g}^{\frac{1}{2}}X_{g}^{T}(\bm{Y}-\bm{X}_{(g)}\bm{\beta}_{(g)})\|_{2}^{2}\right\} }.
\end{align*}
\end{itemize}
\begin{rem} $\bm{Y}-\bm{X}_{(g)}\bm{\beta}_{(g)}$ is the residual
vector when we exclude the $g$th factor $\bm{\beta}_{g}$ in our
regression model. Each element of $X_{g}^{T}(\bm{Y}-\bm{X}_{(g)}\bm{\beta}_{(g)})$
is proportional to the correlation between the each covariate in the
$g$th group and this residual vector. \end{rem}
\begin{itemize}
\item Let $\alpha_{g}^{2}=\frac{1}{\tau_{g}^{2}},\ g=1,2,\dots,G$. Then
\begin{align*}
\alpha_{g}^{2}|\text{rest}\sim\begin{cases}
\text{Inverse Gamma}\left(\text{shape}=\frac{m_{g}+1}{2},\text{scale}=\frac{\lambda^{2}}{2}\right),\text{ if}\;\bm{\beta}_{g}=0,\\
\text{Inverse Gaussian}\left(\frac{\lambda\sigma}{\|\beta_{g}\|_{2}},\lambda^{2}\right),\text{ if}\;\bm{\beta}_{g}\neq0.
\end{cases}
\end{align*}
\item
\begin{align*}
\sigma^{2}|rest
\sim\text{Inverse Gamma}\biggl(&\frac{n}{2}+\frac{1}{2}\sum_{g=1}^{G}m_{g}Z_{g}+\alpha,\\
&\frac{1}{2}[(\bm{Y}-\bm{X}\bm{\beta})^{T}(\bm{Y}-\bm{X}\bm{\beta})+\bm{\beta}^{T}\bm{D}_{\tau}^{-1}\bm{\beta}]+\gamma\biggr)
\end{align*}
\qquad
where $Z_{g}=\begin{cases}
1,\text{ if}\;\bm{\beta}_{g}\neq\bm{0},\\
0,\text{ if}\;\bm{\beta}_{g}=\bm{0},
\end{cases}\ \bm{D}_{\tau}=\mathrm{diag}\{\tau_{1}^{2},\tau_{2}^{2},\ldots,\tau_{G}^{2}\}$.
\item
\begin{align*}
\pi_{0}|\text{rest}\sim\text{Beta}\left(a+\sum_{g=1}^{G}Z_{g},b+\sum_{g=1}^{G}m_{g}-\sum_{g=1}^{G}Z_{g}\right).
\end{align*}
\end{itemize}

\section{Bi-level Selection\label{sec:bi-level-selection}}

We have introduced BGL-SS for group level variable selection in the
last section but it is not always suitable for the problem. In many
applications, it may be desirable to select variables at both the
group level and the individual level. In a genetic association study
\citep{huang_selective_2012}, for example, genetic variations in
the same gene form a natural group. But one genetic variation related
to the disease does not necessarily mean that all the other variations
in the same gene are also associated with the disease. We propose
methods for selecting variables simultaneously at both levels in this section.

\subsection{Bayesian Sparse Group Lasso (BSGL)\label{sub:BSGL}}

\subsubsection{Model Formulation}

With a combination of $L_{1}$- and $L_{2}$-penalty, the sparse group
lasso \citep{simon_sparse-group_2012} has the desirable property of
both group-wise sparsity and within group sparsity. Assuming the following
independent multivariate priors on each group of regression coefficients in \eqref{group_linear_regression},
\begin{align}
\pi\left(\bm{\beta}_{g}\right)\propto\exp\left\{ -\frac{\lambda_{1}}{2\sigma^{2}}\|\bm{\beta}_{g}\|_{1}-\frac{\lambda_{2}}{2\sigma^{2}}\|\bm{\beta}_{g}\|_{2}\right\} ,\quad g=1,2,\dots,G,\label{eq:psql}
\end{align}
then the sparse group lasso estimator in \eqref{eq:sgl} is equivalent
to the MAP solution under this prior.

To find a Bayesian representation of the sparse group lasso where all
posterior conditionals are of standard form and thus greatly simplify
computation, we follow the approach of \citet{park_bayesian_2008}
and \citet{kyung_penalized_2010}, and express the prior as a two
level hierarchical structure including independent $\bm{0}$ mean
Gaussian priors on $\bm{\beta}_{g}$'s with parameters $\bm{\tau}_{g},\gamma_{g}$
and hyperpriors on $\bm{\tau}_{g},\gamma_{g}$.

To enable shrinkage both at the group level and within a group, we propose
the following Bayesian hierarchical model which we refer to as Bayesian
sparse group lasso (BSGL).
\begin{align}
\bm{Y}|\bm{\beta},\sigma^{2} & \sim N\left(\bm{X}\bm{\beta},\sigma^{2}\bm{I}_{n}\right),\\
\bm{\beta}_{g}|\bm{\tau}_{g},\gamma_{g},\sigma^{2} & \sim N\left(\bm{0},\sigma^{2}\bm{V}_{g}\right),\quad g=1,\dots,G,\label{eq:sgl_prior_L1}
\end{align}
where
$\bm{V}_{g}=\mathrm{diag}\{(\frac{1}{\tau_{gj}^{2}}+\frac{1}{\gamma_{g}^{2}})^{-1},\ j=1,2,\dots,m_{g}\}$.
Then we place the following multivariate prior on $\bm{\tau}_{g},\gamma_{g}$
\begin{align}
\pi\left(\tau_{g1}^{2},\dots,\tau_{gm_{g}}^{2},\gamma_{g}^{2}\right)  ={}&c_{g}\left(\lambda_{1}^{2},\lambda_{2}^{2}\right)\prod_{j=1}^{m_{g}}\left[\left(\tau_{gj}^{2}\right)^{-\frac{1}{2}}\left(\frac{1}{\tau_{gj}^{2}}+\frac{1}{\gamma_{g}^{2}}\right)^{-\frac{1}{2}}\right]\left(\gamma_{g}^{2}\right)^{-\frac{1}{2}}\label{eq:sgl_prior_on_variances}\\
 & \times\exp\left\{ -\frac{\lambda_{1}^{2}}{2}\sum_{j=1}^{m_{g}}\tau_{gj}^{2}-\frac{\lambda_{2}^{2}}{2}\gamma_{g}^{2}\right\} .\nonumber
\end{align}
Although this prior has a complicated form and an unknown normalizing
constant depending on $\lambda_{1}$ and $\lambda_{2}$, all the resulting
full conditionals in the Gibbs sampler are standard distributions
and thus are easy and fast to sample from. The propriety of the prior
given in \eqref{eq:sgl_prior_on_variances} is proved in the appendix.

With above hierarchical priors, the marginal prior on $\bm{\beta}_{g}$
is
\[
\pi\left(\bm{\beta}_{g}|\sigma^{2}\right)\propto\exp\left\{ -\frac{\lambda_{1}}{\sigma}\|\bm{\beta}_{g}\|_{1}-\frac{\lambda_{2}}{\sigma}\|\bm{\beta}_{g}\|_{2}\right\} ,
\]
which is a prior of the form \eqref{eq:psql} with our two level hierarchical
prior specification.

\subsubsection{Hyperparameter Specification}

The specification of hyperparameters $\lambda_{1}^{2},\lambda_{2}^{2}$
is very important because it expresses our prior belief of sparsity
and the amount of shrinkage. We place a hyper-prior on them instead
of imposing fixed values. Define
$C
(\lambda_{1}^{2},\lambda_{2}^{2})=\prod_{g=1}^{G}c_{g}
(\lambda_{1}^{2},\lambda_{2}^{2})$.
The following prior is assigned to $\lambda_{1}^{2}$ and $\lambda_{2}^{2}$,
\[
p\left(\lambda_{1}^{2},\lambda_{2}^{2}\right)\propto C^{-1}\left(\lambda_{1}^{2},\lambda_{2}^{2}\right)\left(\lambda_{1}^{2}\right)^{p}\left(\lambda_{2}^{2}\right)^{G/2}\exp\left\{ -d_{1}\lambda_{1}^{2}-d_{2}\lambda_{2}^{2}\right\} ,
\]
where $d_{1}>0,d_{2}>0$. It is easy to show that this prior is proper.
To make it a moderately diffuse prior, we specify small values for
$d_{1}$ and $d_{2}$, $d_{1}=d_{2}=10^{-1}$.

\subsection{Bayesian Sparse Group Selection with Spike and Slab Prior (BSGS-SS)\label{sub:BSGL-SS}}

Although the Bayesian sparse group lasso has shrinkage effects at
both the group level and also within a group, it does not produce
sparse model since the posterior mean/median estimators are never
exact 0. To achieve sparsity at both levels for variable selection purpose, and to improve out-of-sample
prediction performance, we propose the Bayesian Sparse Group Selection
with Spike and Slab prior (BSGS-SS), which utilizes spike and slab
type priors for both group variable selection and individual variable
selection. The difficulty of this problem lies in how to introduce
both types of sparsity with spike and slab priors.

\subsubsection{Model Specification}

We reparametrize the coefficients vectors to tackle the two kinds
of sparsity separately:
\begin{align}
\bm{\beta}_{g}=\bm{V}_{g}^{\frac{1}{2}}\bm{b}_{g},
\ \text{where }\bm{V}_{g}^{\frac{1}{2}}=\mathrm{diag}\left\{ \tau_{g1},\dots,\tau_{gm_{g}}\right\},\ \tau_{gj}\geq0,\ g=1,\dots,G;\ j=1,\dots,m_{g},
\end{align}
where $\bm{b}_{g}$, when nonzero, has a $\bm{0}$ mean multivariate
normal distribution with identity matrix as its covariance matrix.
Thus the diagonal elements of $\bm{V}_{g}^{\frac{1}{2}}$ control
the magnitude of elements of $\bm{\beta}_{g}$. To select variables
at the group level, we assume the following multivariate spike and
slab prior for each $\bm{b}_{g}$:
\begin{align}
\bm{b}_{g}\stackrel{ind}{\sim}\left(1-\pi_{0}\right)\bm{N}_{m_{g}}\left(\bm{0},\bm{I}_{m_{g}}\right)+\pi_{0}\delta_{0}\left(\bm{b}_{g}\right),\quad g=1,\dots,G.
\end{align}
Note that when $\tau_{gj}=0$, $\beta_{gj}$ is essentially dropped
out of the model even when $b_{gj}\neq0$. So in order to choose variables
within each relevant group, we assume the following spike and slab
prior for each $\tau_{gj}$:
\begin{align}
\tau_{gj}\stackrel{ind}{\sim}\left(1-\pi_{1}\right)N^{+}\left(0,s^{2}\right)+\pi_{1}\delta_{0}\left(\tau_{gj}\right),\quad g=1,\dots,G;\ j=1,\dots,m_{g},
\end{align}
where $N^{+}(0,s^{2})$ denotes a normal $N(0,s^{2})$
distribution truncated below at 0. Note that this truncated normal
distribution has mean $\sqrt{\frac{2}{\pi}}s$ and variance $s^{2}$.
\begin{rem} If $m_{g}=1$, $\beta_{g}=\tau_{g}b_{g}$ is a scalar,
and still has a spike and slab distribution. The prior probability
of $\beta_{g}=0$ is $1-(1-\pi_{0})(1-\pi_{1})$,
which is larger than both $\pi_{0}$ and $\pi_{1}$, but smaller than
$\pi_{0}+\pi_{1}$. As a comparison, the sparse group lasso penalty
for the $g$th group of coefficients becomes $(\lambda_{1}+\lambda_{2})\|\beta_{g}\|_{1}$
when $m_{g}=1$. Thus the penalty parameter is the sum of the individual
level penalty parameter $\lambda_{1}$, and the group level penalty
parameter $\lambda_{2}$. \end{rem}\smallskip

\begin{rem} Alternatively, we could enforce both types of sparsity
by generalizing the binary masking model of \citet{kuo_variable_1998}.
We can reparameterize the regression coefficients as $\beta_{gj}=\gamma_{g}^{(1)}\gamma_{gj}^{(2)}b_{gj}$, where $\gamma_{g}^{(1)}$ is a binary indicator of whether
the $g$th group of coefficients are all 0, and $\gamma_{gj}^{(2)}$
indicates whether $\beta_{gj}=0$. The following priors are assumed:
\begin{align*}
\gamma_{g}^{\left(1\right)} & \sim\text{Bernoulli}\left(\pi_{0}\right),\quad g=1,\dots,G,\\
\gamma_{gj}^{\left(2\right)} & \sim\text{Bernoulli}\left(\pi_{1}\right),\quad g=1,\dots,G;\ j=1,\dots,m_{g},\\
b_{gj} & \sim N\left(0,s^{2}\right),\quad g=1,\dots,G;\ j=1,\dots,m_{g}.
\end{align*}
We expect that the above alternative formulation to have comparable
performance with the BSGS-SS model that we proposed. \citet{stingo_incorporating_2011}
also uses two sets of binary indicators for group and individual level
selection for a more specific group selection problem, in which groups
may be overlapping and certain dependence structure among variables
exists. \end{rem}

Instead of specifying fixed values for hyperparameters, typical non-informative
priors are used. We assume an inverse gamma prior for the error variance
$\sigma^{2}$, where shape and scale parameters are chosen to be relatively
small:
\begin{align}
\sigma^{2}\sim\text{Inverse Gamma}\left(\alpha,\gamma\right),\quad \alpha=0.1,\ \gamma=0.1.
\end{align}
To decide the values of hyperparameters $\pi_{0},\pi_{1}$, we assume
conjugate beta hyper-priors:
\begin{equation}
\pi_{0}\sim\text{Beta}\left(a_{1},a_{2}\right),\quad  \pi_{1}\sim\text{Beta}\left(c_{1},c_{2}\right).
\end{equation}
For $s^{2}$, we place a conjugate inverse gamma prior on it,
\[
s^{2}\sim\text{Inverse Gamma}\left(1,t\right),
\]
and estimate $t$ with the Monte Carlo EM algorithm \citep{casella_empirical_2001,park_bayesian_2008}.
For the $k$th EM update,
\[
t^{\left(k\right)}=\frac{1}{E_{t^{\left(k-1\right)}}\left[\frac{1}{s^{2}}\mid\bm{Y}\right]},
\]
where the posterior expectation of $\frac{1}{s^{2}}$ is estimated
from the Gibbs samples based on $t^{(k-1)}$.

Therefore, with the above model specification, the joint posterior
of $\bm{b},\tau^{2},\sigma^{2},\pi_{0},\pi_{1}$ conditional on observed
data is
\begin{align*}
&{}p\left(\bm{b},\bm{\tau}^{2},\sigma^{2},\pi_{0},\pi_{1},s^{2}\mid\bm{Y},\bm{X}\right)\\
&\quad{}\propto \left(\sigma^{2}\right)^{-\frac{n}{2}}\exp\left\{ -\frac{1}{2\sigma^{2}}
\biggl\|\bm{Y}-\sum_{g=1}^{G}\bm{X}_{g}\bm{V}_{g}^{\frac{1}{2}}\bm{b}_{g}\biggr\|_{2}^{2}\right\} \\
&\quad{}\times  \prod_{g=1}^{G}\left[\left(1-\pi_{0}\right)\left(2\pi\right)^{-\frac{mg}{2}}\exp\left\{ -\frac{1}{2}\bm{b}_{g}^{T}\bm{b}_{g}\right\} I\left[\bm{b}_{g}\neq0\right]+\pi_{0}\delta_{0}\left(\bm{b}_{g}\right)\right]\\
&\quad{}\times  \prod_{g=1}^{G}\prod_{j=1}^{m_{g}}\left[\left(1-\pi_{1}\right)\cdot2\left(2\pi s^{2}\right)^{-\frac{1}{2}}\exp\left\{ -\frac{\tau_{gj}^{2}}{2s^{2}}\right\} I\left[\tau_{gj}>0\right]+\pi_{1}\delta_{0}\left(\tau_{gj}\right)\right]\\
&\quad{}\times  \left(\sigma^{2}\right)^{-\alpha-1}\exp\left\{ -\frac{\gamma}{\sigma^{2}}\right\} \\
&\quad{}\times  \pi_{0}^{a_{1}-1}\left(1-\pi_{0}\right)^{a_{1}-1}\\
&\quad{}\times  \pi_{1}^{c_{1}-1}\left(1-\pi_{1}\right)^{c_{1}-1}\\
&\quad{}\times  t\left(s^{2}\right)^{-2}\exp\left\{ -\frac{t}{s^{2}}\right\}.
\end{align*}

\subsubsection{Gibbs Sampler}

Similar to Subsection \ref{sub:BGL-SS-Sampler}, we define the coefficients
vector without the $j$th element in the $g$th group as
\begin{align*}
\bm{\beta}_{(gj)}=\left(\beta_{11},\dots,\beta_{1m_{1}},\dots,\beta_{g1},\dots,\beta_{g,j-1},\beta_{g,j+1},\dots,\beta_{gm_{g}},\dots,\beta_{Gm_{G}}\right)^{T},
\end{align*}
and the covariates matrix corresponding to $\bm{\beta}_{(gj)}$
as
\begin{align*}
\bm{X}_{(gj)}=\left(\bm{x}_{11},\dots,\bm{x}_{1m_{1}},\dots,\bm{x}_{g1},\dots,\bm{x}_{g,j-1},\bm{x}_{g,j+1},\dots,\bm{x}_{gm_{g}},\dots,\bm{x}_{Gm_{G}}\right).
\end{align*}

\begin{itemize}
\item The posterior distribution of $\bm{b}_{g}$ conditioning on everything
else is still a multivariate spike and slab distribution,
\begin{align*}
\bm{b}_{g}\mid\text{rest}\sim l_{g}\delta_{0}\left(\bm{b}_{g}\right)+\left(1-l_{g}\right)\bm{N}_{m_{g}}\left(\bm{\mu}_{g},\bm{\Sigma}_{g}\right) & ,
\end{align*}
where $l_{g}$, the posterior probability of $\bm{b}_{g}$ equal to
$\bm{0}$ given the remaining parameters, is
\begin{align*}
l_{g} & =P\left(\bm{b}_{g}=\bm{0}|\text{rest}\right)\\
 & =\frac{\pi_{0}}{\pi_{0}+\left(1-\pi_{0}\right)\mid\bm{\Sigma}\mid^{\frac{1}{2}}\exp\left\{ \frac{1}{2\sigma^{4}}\parallel\bm{\Sigma}_{g}^{\frac{1}{2}}\bm{V}_{g}^{\frac{1}{2}}\bm{X}_{g}^{T}\left(\bm{Y}-\bm{X}_{\left(g\right)}V_{\left(g\right)}^{\frac{1}{2}}\bm{b}_{\left(g\right)}^{\frac{1}{2}}\right)\parallel_{2}^{2}\right\} },
\end{align*}
$\bm{\mu}_{g}=\frac{1}{\sigma^{2}}\bm{\Sigma}_{g}^{\frac{1}{2}}\bm{V}_{g}^{\frac{1}{2}}\bm{X}_{g}^{T}
(\bm{Y}-\bm{X}_{(g)}V_{(g)}^{\frac{1}{2}}\bm{b}_{(g)}^{\frac{1}{2}})$,
and $\bm{\Sigma}_{g}=(\bm{I}_{m_{g}}+\frac{1}{\sigma^{2}}\bm{V}_{g}^{\frac{1}{2}}\bm{X}_{g}^{T}\bm{X}_{g}\bm{V}_{g}^{\frac{1}{2}})^{-1}$.
\item The conditional posterior of $\tau_{gj}$ is a spike and slab distribution,
with the slab a positive part normal distribution:
\begin{align*}
\tau_{gj}\mid\text{rest}\sim q_{gj}\delta_{0}\left(\tau_{gj}\right)+\left(1-q_{gj}\right)
\bm{N}^{+}\left(u_{gj},v_{gj}^{2}\right),\  g=1,2,\dots,G;\,j=1,2,\dots,m_{G},
\end{align*}
where $u_{gj}=\frac{1}{\sigma^{2}}v_{gj}^{2}
(\bm{Y}-
\bm{X}_{(gj)}
\bm{\beta}_{(gj)}
)^{T}\bm{X}_{gj}b_{gj}$,
$v_{gj}^{2}=
(\frac{1}{s^{2}}+\frac{1}{\sigma^{2}}\bm{X}_{gj}^{T}\bm{X}_{gj}b_{gj}^{2}
)^{-1}$
and
\begin{align*}
q_{gj}=p\left(\tau_{gj}=0\mid \text{rest}\right)=\frac{\pi_{1}}{\pi_{1}+2\left(1-\pi_{1}\right)\left(s^{2}\right)^{-\frac{1}{2}}\left(v_{gj}^{2}\right)^{\frac{1}{2}}\exp\left\{ \frac{u_{gj}^{2}}{2v_{gj}^{2}}\right\} \left[\Phi\left(\frac{u_{gj}}{v_{gj}}\right)\right]}.
\end{align*}
\item
\begin{align*}
\sigma^{2}|\text{rest}\sim\text{Inverse Gamma}\left(\frac{n}{2}+\alpha,\frac{1}{2}\parallel\bm{Y}-\bm{X}\bm{\beta}\parallel_{2}^{2}+\gamma\right).
\end{align*}
\item With conjugate Beta priors, the posteriors of $\pi_{0}$ and $\pi_{1}$
conditional on everything else continue to be Beta distributions:
\begin{align*}
\pi_{0}\mid\text{rest}&\sim\text{Beta}\left(\#\left(\bm{b}_{g}=0\right)+a_{1},\#\left(\bm{b}_{g}\neq0\right)+a_{2}\right),\\
\pi_{1}\mid\text{rest}&\sim\text{Beta}\left(\#\left(\tau_{gj}=0\right)+c_{1},\#\left(\tau_{gj}\neq0\right)+c_{2}\right).
\end{align*}

\item With conjugate inverse gamma prior, the conditional posterior of $s^{2}$
is still an inverse gamma distribution:
\[
s^{2}\mid\text{rest}\sim\text{Inverse Gamma}\left(1+\frac{1}{2}\#\left(\tau_{gj}=0\right),t+\frac{1}{2}\sum_{g,j}\tau_{gj}^{2}\right).
\]

\end{itemize}

\section{Simulation\label{sec:Simulation}}

\noindent We simulate data from the following true model:
\begin{align*}
\bm{Y}=\bm{X}\bm{\beta}+\bm{\epsilon},\text{ where }\epsilon_{i}\stackrel{iid}{\thicksim}N(0,\sigma^{2}),\ i=1,2,\dots,n.
\end{align*}
For the following examples, we compare the variable selection accuracy
and prediction performance of BGL-SS, BSGL, BSGS-SS with 4 other models:
linear regression, the Group Lasso (GL), the Sparse Group Lasso (SGL) and the Bayesian Group
Lasso (BGL), when applicable. Five examples are considered in our
simulations. The third one is from the original lasso paper \citep{tibshirani_regression_1996}.
\begin{itemize}
\item Example 1. We simulate a data set with 100 observations and 20 covariates,
which are divided into 4 groups with 5 covariates each. We randomly
sample 60 observations to train the model and use the remaining 40
to compare the prediction performance of proposed model with other
lasso variations. Let
\begin{align*}
\bm{\beta}=\left(\left(0.3,-1,0,0.5,0.01\right),\bm{0},\left(0.8,0.8,0.8,0.8,0.8\right),\bm{0}\right),
\end{align*}
where $\bm{0}$ is the 0 vector of length 5. The pairwise correlation
between covariates $x_{i}$ and $x_{j}$ is 0.5 for $i\neq j$. We
specify $\sigma=3$.
\item \noindent Example 2. This example is a large $p$ small $n$ problem
with $n=60$ and $p=80$. 40 observations are randomly sampled to train
the model and the remaining 20 are used to compare the prediction
performance. 80 predictors are grouped into 16 groups of 5 covariates
each. We define the $j$th predictor in group $g$ as $X_{gj}=z_{g}+z_{gj}$,
where $z_{g}$ and $z_{gj}$ are independent standard normal variates,
$g=1,\dots,16;\ j=1,2,\dots,5$. Thus predictors within a group are correlated
with pairwise correlation $\frac{1}{2}$ while the predictors in different
groups are independent. Let
\begin{align*}
\bm{\beta}=\left(\left(1,2,3,4,5\right),\bm{0},\left(0.1,0.2,0.3,0.4,0.5\right),\bm{0},\bm{0},\bm{0},\bm{0},\bm{0},\bm{0},\bm{0},\bm{0},\bm{0},\bm{0},\bm{0},\bm{0},\bm{0}\right)
\end{align*}
where $\bm{0}$ is the 0 vector of length 5. We use $\sigma=2$.
\item \noindent Example 3. In this example, we simulate a data set with
$n=100$ and $p=40$. 60 observations are used to train the model and the
remaining 40 are used for testing the predictions. Let $\bm{\beta}=(\bm{0},\bm{2},\bm{0},\bm{2})$,
where $\bm{0}$ and $\bm{2}$ are both of length 10, with all elements
0 or 2, respectively. We simulate predictors in the same way as in
Example 2 except for necessary dimension changes. The error standard
deviation $\sigma$ is 2.
\item \noindent Example 4. This example is the same as Example 3 except
the true coefficients
\[
\bm{\beta}=\left(\bm{0},\left(2,2,2,2,2,0,0,0,0,0\right),\bm{0},\left(2,2,2,2,2,0,0,0,0,0\right)\right),
\]
where $\bm{0}$ is a 0 vector of length 10. So this example, like
Example 1, has sparsity at the group level and also sparsity within
nonzero groups.
\item \noindent Example 5. This example is taken from \citet{yuan_model_2006}.
$Z_{1},Z_{2},\dots,Z_{20}$ and $W$ were independently generated from
the standard normal distribution, and we define $X_{i}=\frac{(Z_{i}+W)}{\sqrt{2}}$.
The first 10 covariates are each expanded to a third order polynomial
thus we have 10 factors consisting of third order polynomial terms.
The last 10 covariates are each trichotomized as $0,1,2$ if it is
smaller than $\phi^{-1}(1/3)$, larger than $\phi^{-1}(2/3)$,
or between them. The simulation model is
\[
Y=\left(X_{3}+X_{3}^{2}+X_{3}^{3}\right)+\left(\frac{2}{3}X_{6}-X_{6}^{2}+\frac{1}{3}X_{6}^{3}\right)+2I\left(X_{11}=0\right)+I\left(X_{11}=1\right)+\epsilon
\]
where $\epsilon\sim N(0,2^{2})$. We simulate 200 samples
and use 100 for training and the rest 100 for testing. We have 20
factors with 50 covariates in total.
\end{itemize}
SPArse Modeling Software (SPAM) is the most stable program we have
found for fitting group lasso and sparse group lasso \citep[see][]{mairal_network_2010,jenatton_proximal_2011},
and we use 5-fold cross-validation to choose optimal $\lambda$s.
For BGL-SS, we have conjugate prior on $\pi_{0}$, so we only need
to specify suitable hyperparameters, which we choose $a=1,b=1$. For
BSGS-SS, we have beta priors on both $\pi_{0}$ and $\pi_{1}$, and
we set $a_{1}=a_{2}=c_{1}=c_{2}=1$. For Bayesian models, we generate
from the full posterior distribution with a Gibbs Sampler running
10000 iterations in which the first 5000 are burn-ins. Posterior mean
and posterior median are both used as our Bayes estimators and we
will compare their variable selection and prediction performance.
To summarize the prediction errors, we calculate the median mean squared
error in 50 simulations.

\begin{table}[ht]
\centering
\tabcolsep=0pt
\begin{tabular*}{100mm}{@{\extracolsep{\fill}}lrrrrrr}
\hline
& \multicolumn{2}{c}{BGL-SS} & \multicolumn{2}{c}{BSGS-SS} & \multirow{2}{*}{GL}  & \multirow{2}{*}{SGL} \\
\cline{2-3} \cline{4-5}
& MTM & HPPM & MTM & HPPM &  &  \\
  \hline
$Example\text{ }1$ \\
TPR & 0.96 & 0.98 & 0.79 & 0.89 & 0.97 & 0.90 \\
  FPR & 0.23 & 0.48 & 0.09 & 0.19 & 0.65 & 0.53 \\
$Example\text{ }2$ \\
  TPR & 0.90 & 0.91 & 0.82 & 0.92 & 0.98 & 0.87 \\
  FPR & 0.06 & 0.12 & 0.02 & 0.02 & 0.39 & 0.16 \\
$Example\text{ }3$ \\
  TPR & 1.00 & 1.00 & 1.00 & 1.00 & 1.00 & 1.00 \\
  FPR & 0.00 & 0.00 & 0.02 & 0.03 & 0.44 & 0.26 \\
$Example\text{ }4$ \\
  TPR & 1.00 & 1.00 & 1.00 & 1.00 & 1.00 & 1.00 \\
  FPR & 0.34 & 0.34 & 0.22 & 0.34 & 0.79 & 0.32 \\
$Example\text{ }5$ \\
  TPR & 0.97 & 0.99 & 0.91 & 0.94 & 0.99 & 0.94 \\
  FPR & 0.14 & 0.54 & 0.02 & 0.02 & 0.40 & 0.30 \\
   \hline
\end{tabular*}
\caption{Mean True/False Positive Rate for six methods in five simulation examples, based on 50 simulations.}
\label{table:var_sel}
\end{table}

In Table \ref{table:var_sel}, we summarize the model selection accuracy
of different methods. For both BGL-SS and BSGS-SS, the median thresholding
model (MTM) and the highest posterior probability model (HPPM) are
compared by true and false positive rate. We also list the group lasso
and sparse group lasso results for comparison. Median thresholding
model, which is more parsimonious, outperforms all other methods including
the corresponding highest posterior probability model. The group lasso
and the sparse group lasso with penalty parameters chosen by cross validation
tend to select much more variables than our spike and slab methods.
\citet{leng_note_2004} showed that when the tuning parameter is selected
by minimizing the prediction error, the lasso procedure is inconsistent
in variable selection in general. It is suspected \citep{wang_note_2008}
that the group lasso may suffer the same variable selection inconsistency
which may explain why the group lasso and the sparse group lasso tends to
select more variables and have higher false positive rate in our simulation.
On the other hand, model selected by median thresholding has very
low false positive rate and even outperforms the gold standard of
Bayesian variable selection -- the highest posterior probability model.

Table \ref{table:prediction} summarizes the median mean squared prediction
error for all 5 simulated examples using 9 methods to fit the simulated
data, based on 50 replications. The bootstrapped standard errors of
the medians are given in the parentheses. A couple of observations
can be made from Table \ref{table:prediction}:
\begin{itemize}
\item BGL-SS is comparable with the group lasso in prediction except in
Example 2, and BSGS-SS outperforms the sparse group lasso in all examples;
\item Posterior mean estimator and posterior median estimator have very
close prediction error;
\item BGL and BSGL does not predict as well as their frequentist counterpart,
GL and SGL;
\item When there is no obvious sparsity within relevant groups, BGL-SS usually
performs favorably or sometimes better than BSGS-SS; but when there
is significant sparsity within relevant groups (Example 4), BSGS-SS
is very good at identifying within group sparsity and thus further
improves the prediction performance from BGL-SS;
\item The fact that BGL-SS does not predict well in Example 2 suggests that
a flat prior with mean $\frac{1}{2}$ on $\pi_{0}$ does not work
well for high-dimensional problems in which most groups of predictors
are 0. We note that it still works much better than the group lasso
in terms of variable selection even with this flat prior.
\end{itemize}
Now we demonstrate the sensitivity of BGL-SS for model selection to
the specification of $\pi_{0}$. We fix $\pi_{0}$ at 0.2, 0.5, 0.8
and assume $\text{Beta}(0.5,0.5)$, $\text{Beta}(1,1)$,
$\text{Beta}(1.5,1.5)$ priors. Table \ref{table:pi_sensitivity}
shows that the misclassification error, the percentage of misclassified
variables, of the median thresholding model and the highest probability
model with different specification of $\pi_{0}$. For comparison we
append the result of the group lasso, with penalty parameter chosen
by cross-validation, in the last row. For all choices of $\pi_{0}$,
the median thresholding model is very stable and misclassifies at
most three variables, while the highest probability model is very
sensitive to the choice of $\pi_{0}$. We also note that although
the misclassification error of the group lasso is much higher, its
prediction error is comparable to the BGL-SS in this example as we
have seen in Table \ref{table:prediction}.

\begin{table}[t!]
\fontsize{9.5}{11.5}\selectfont
\centering
\tabcolsep=0pt
\begin{tabular*}{\textwidth}{@{\extracolsep{\fill}}lrrrrr@{}}
  \hline
 & Example 1 & Example 2 & Example 3 & Example 4 & Example 5 \\
  \hline
BGL-SS with mean & 9.69(0.35) & 6.79(0.39) & 6.45(0.29) & 6.41(0.34) & 5.24(0.17) \\
  BGL-SS with median & 9.76(0.40) & 6.60(0.43) & 6.46(0.25) & 6.40(0.32) & 5.08(0.18) \\
  BSGS-SS with mean & 10.07(0.38) & 5.51(0.21) & 6.83(0.42) & 5.37(0.15) & 4.83(0.16) \\
  BSGS-SS with median & 10.37(0.34) & 5.59(0.32) & 6.51(0.38) & 5.38(0.12) & 4.92(0.15) \\
  Group Lasso & 9.82(0.51) & 5.99(0.33) & 5.91(0.38) & 6.98(0.46) & 5.30(0.16) \\
  Sparse Group Lasso & 10.48(0.55) & 5.75(0.45) & 6.88(0.34) & 5.90(0.28) & 5.22(0.23) \\
  Bayesian Group Lasso & 10.53(0.34) & 8.24(0.51) & 7.89(0.24) & 7.48(0.41) & 6.46(0.23) \\
  Bayesian Sparse Group lasso & 10.08(0.47) & 10.55(0.56) & 10.21(0.37) & 8.65(0.41) & 6.03(0.16) \\
  Linear Regression & 11.19(0.42) & -- & 12.71(0.96) & 12.68(1.03) & 8.71(0.54) \\
   \hline
\end{tabular*}
\caption{Median mean squared error for nine
methods in five simulation examples, based on 50 replications.}
\label{table:prediction}
\end{table}

\begin{table}[t!]
\centering
\tabcolsep=0pt
\begin{tabular*}{116mm}{@{\extracolsep{\fill}}lcccccc}
\hline
& \multicolumn{3}{c}{Fix $\pi_0$} & \multicolumn{3}{c}{Hyperprior} \\
\cline{2-4} \cline{5-7}
& 0.20 & 0.50 & 0.80 & $a=b=0.50$ & $a=b=1.00$ & $a=b=1.50$ \\
   \hline
MTM & 0.10 & 0.05 & 0.10 & 0.15 & 0.10 & 0.10 \\
HPPM & 0.30 & 0.05 & 0.05 & 0.55 & 0.55 & 0.30 \\
GL &  0.55 & 0.55 & 0.55 & 0.55 & 0.55 & 0.55 \\
   \hline
\end{tabular*}
\caption{Sensitivity analysis for BGL-SS using Example 1.}
\label{table:pi_sensitivity}
\end{table}

Posterior mean and median estimators of our spike and slab models
are compared in Table~\ref{table:coef}. Two variations of Example
1, with $\sigma=1$ or $\sigma=3$, respectively, are both fitted by
BGL-SS and BSGS-SS model. For both cases, the posterior median estimators
both produce 0 estimates and correctly identify the two most important
factors. When the signal-to-noise ratio is high, the posterior mean
estimates shrink coefficients of redundant variables to very small
values. But when there is too much noise, posterior mean does not
have enough shrinkage effects to help us with variable selection.
Regarding within group sparsity, $\beta_{3}$ was shrunk to 0 by BSGS-SS
at the cost of shrinking $\beta_{5}$, which has a very small true
value, 0.01.

\begin{table}[t!]
\centering
\tabcolsep=0pt
\begin{tabular*}{\textwidth}{@{\extracolsep{\fill}}rrrrrrrrrr}
\hline
&      & \multicolumn{4}{c}{$\sigma=3$} & \multicolumn{4}{c}{$\sigma=1$} \\
\cline{3-6} \cline{7-10}
&      & \multicolumn{2}{c}{BGL-SS} & \multicolumn{2}{c}{BSGS-SS} & \multicolumn{2}{c}{BGL-SS} & \multicolumn{2}{c}{BSGS-SS} \\
\cline{3-4} \cline{5-6} \cline{7-8} \cline{9-10}
& True & Median & Mean & Median & Mean & Median & Mean & Median & Mean \\
\hline
$\beta_{1}$ & 0.3 & 0.09 & 0.127 & 0 & 0.115 & 0.291 & 0.289 & 0.288 & 0.282 \\
  $\beta_{2}$ & $-$1 & $-$0.611 & $-$0.633 & $-$0.656 & $-$0.666 & $-$1.103 & $-$1.056 & $-$1.274 & $-$1.27 \\
  $\beta_{3}$ & 0 & $-$0.056 & $-$0.086 & 0 & $-$0.03 & $-$0.065 & $-$0.065 & 0 & $-$0.022 \\
  $\beta_{4}$ & 0.5 & 0.619 & 0.637 & 0.861 & 0.812 & 0.695 & 0.675 & 0.789 & 0.788 \\
  $\beta_{5}$ & 0.01 & 0 & 0.011 & 0 & 0.014 & $-$0.01 & $-$0.008 & 0 & 0.003 \\
  $\beta_{6}$ & 0 & 0 & 0.158 & 0 & 0.111 & 0 & 0.006 & 0 & 0.011 \\
  $\beta_{7}$ & 0 & 0 & $-$0.144 & 0 & $-$0.071 & 0 & $-$0.005 & 0 & $-$0.008 \\
  $\beta_{8}$ & 0 & 0 & 0.11 & 0 & 0.067 & 0 & 0.004 & 0 & 0.007 \\
  $\beta_{9}$ & 0 & 0 & $-$0.183 & 0 & $-$0.071 & 0 & $-$0.007 & 0 & $-$0.014 \\
  $\beta_{10}$ & 0 & 0 & 0.165 & 0 & 0.105 & 0 & 0.006 & 0 & 0.014 \\
  $\beta_{11}$ & 0.8 & 1.534 & 1.522 & 1.555 & 1.555 & 1.237 & 1.232 & 1.23 & 1.231 \\
  $\beta_{12}$ & 0.8 & 0.271 & 0.279 & 0.053 & 0.187 & 0.696 & 0.693 & 0.689 & 0.685 \\
  $\beta_{13}$ & 0.8 & 0.877 & 0.876 & 0.728 & 0.709 & 0.948 & 0.952 & 0.906 & 0.905 \\
  $\beta_{14}$ & 0.8 & 0.73 & 0.737 & 0.66 & 0.666 & 0.956 & 0.942 & 1.055 & 1.053 \\
  $\beta_{15}$ & 0.8 & 0.744 & 0.741 & 0.527 & 0.532 & 0.928 & 0.926 & 0.919 & 0.917 \\
  $\beta_{16}$ & 0 & 0 & $-$0.128 & 0 & $-$0.059 & 0 & $-$0.002 & 0 & $-$0.003 \\
  $\beta_{17}$ & 0 & 0 & 0.111 & 0 & 0.078 & 0 & 0.002 & 0 & 0.006 \\
  $\beta_{18}$ & 0 & 0 & 0.177 & 0 & 0.131 & 0 & 0.003 & 0 & 0.006 \\
  $\beta_{19}$ & 0 & 0 & $-$0.023 & 0 & $-$0.003 & 0 & $-$0.001 & 0 & $-$0.003 \\
  $\beta_{20}$ & 0 & 0 & $-$0.056 & 0 & $-$0.015 & 0 & $-$0.002 & 0 & $-$0.003 \\
\hline
\end{tabular*}
\caption{Posterior mean and posterior median estimators under spike and slab models using Example~1 with two different error variances.}
\label{table:coef}
\end{table}

\section{Discussion\label{sec:discussion}}

The primary goal of the group lasso is to both select groups of variables
and estimate corresponding coefficients. Previous Bayesian approaches
via multivariate scale mixture of normals do have shrinkage effects
at the group level but do not yield sparse estimators.

Spike and slab type priors facilitate variable selection by putting
a point mass at 0, or in the case of group variable selection, a multivariate
point mass at $\bm{0}_{m\times1}$ for an $m$-dimensional coefficients group. Since the posterior mean estimator
still does not produce sparse estimators, two variable selection criterion
were proposed. Highest posterior probability model \citep{geweke_variable_1994,kuo_variable_1998,george_approaches_1997}
is a very popular one since via Gibbs sampling simulations we could
easily obtain the model and an estimate of its corresponding posterior
probability. Alternatively, one can use FDR based variable selection
which selects variables with marginal inclusion probability larger
than certain threshold and we could choose the threshold to control
the overall average Bayesian FDR rate \citep{bonato_bayesian_2011,zhang_bayesian_2014}.
Median probability model is advocated by \citet{barbieri_optimal_2004}
due to its optimal prediction performance. We note that this is the
special case of FDR based methods with thresholds set to $\frac{1}{2}$.
Our median thresholding model is more parsimonious than the median
probability model because the median of a variable with a spike and
slab distribution is 0 if and only if the probability for it to be
either larger or smaller than 0 are both less than $\frac{1}{2}$.

Posterior median estimator is distinctive in the Bayesian methods
since it can both select and estimate automatically like the lasso
estimator. We demonstrate in this paper that it can achieve superior
variable selection accuracy and good prediction performance at the
same time. It tends to select fewer variables than group lasso methods
but achieves similar or sometimes better prediction error. Compared
to the highest probability model, the median thresholding model is
at least as good as and sometimes better than it in terms of true
and false positive rate.

\appendix

\section{Propriety of \eqref{eq:sgl_prior_on_variances}}

Prior \eqref{eq:sgl_prior_on_variances} is proper since
\begin{align*}
 & \prod_{j=1}^{m_{g}}\left[\left(\tau_{gj}^{2}\right)^{-\frac{1}{2}}\left(\frac{1}{\tau_{gj}^{2}}+\frac{1}{\gamma_{g}^{2}}\right)^{-\frac{1}{2}}\right]\left(\gamma_{g}^{2}\right)^{-\frac{1}{2}}\exp\left\{ -\frac{\lambda_{1}^{2}}{2}\sum_{j=1}^{m_{g}}\tau_{gj}^{2}-\frac{\lambda_{2}^{2}}{2}\gamma_{g}^{2}\right\} \\
 & =\prod_{j=1}^{m_{g}}\left[\left(1+\frac{\tau_{gj}^{2}}{\gamma_{g}^{2}}\right)^{-\frac{1}{2}}\right]\left(\gamma_{g}^{2}\right)^{-\frac{1}{2}}\exp\left\{ -\frac{\lambda_{1}^{2}}{2}\sum_{j=1}^{m_{g}}\tau_{gj}^{2}-\frac{\lambda_{2}^{2}}{2}\gamma_{g}^{2}\right\} \\
 & \leq\left(\gamma_{g}^{2}\right)^{-\frac{1}{2}}\exp\left\{ -\frac{\lambda_{2}^{2}}{2}\gamma_{g}^{2}\right\} \exp\left\{ -\frac{\lambda_{1}^{2}}{2}\sum_{j=1}^{m_{g}}\tau_{gj}^{2}\right\} .
\end{align*}

\section{Marginal Prior for The Bayesian Sparse Group Lasso}

With \eqref{eq:sgl_prior_L1}\eqref{eq:sgl_prior_on_variances}, the
marginal prior on $\bm{\beta}_{g}$ is:
\begin{align*}
\pi\left(\bm{\beta}_{g}|\sigma^{2}\right)
 & \propto\int_{0}^{\infty}\cdots\int_{0}^{\infty}\prod_{j=1}^{m_{g}}\left(\frac{1}{\tau_{gj}^{2}}+\frac{1}{\gamma_{g}^{2}}\right)^{\frac{1}{2}}\exp\left\{ -\frac{1}{2\sigma^{2}}\sum_{j=1}^{m_{g}}\left(\frac{1}{\gamma_{g}^{2}}+\frac{1}{\tau_{gj}^{2}}\right)\beta_{gj}^{2}\right\} \\
 & \times\left(\prod_{j=1}^{m_{g}}\left(\frac{1}{\tau_{gj}^{2}}+\frac{1}{\gamma_{g}^{2}}\right)^{-\frac{1}{2}}(\tau_{gj}^{2})^{-\frac{1}{2}}\right)(\gamma_{g}^{2})^{-\frac{1}{2}}\exp\left\{ -\frac{\lambda_{1}^{2}}{2}\sum_{j=1}^{m_{g}}\tau_{gj}^{2}-\frac{\lambda_{2}^{2}}{2}\gamma_{g}^{2}\right\}\\
 & \times\left(\prod_{j=1}^{m_{g}}d\tau_{gj}^{2}\right)d\gamma_{g}^{2}\\
 &\propto\prod_{j=1}^{m_{g}}\int_{0}^{\infty}\left(\tau_{gj}^{2}\right)^{-\frac{1}{2}}\exp\left\{ -\frac{1}{2}\left(\frac{\beta_{gj}^{2}}{\sigma^{2}}\frac{1}{\tau_{gj}^{2}}+\lambda_{1}^{2}\tau_{gj}^{2}\right)\right\} d\tau_{gj}^{2}\\
 & \times\int_{0}^{\infty}\left(\gamma_{g}^{2}\right)^{-\frac{1}{2}}\exp\left\{ -\frac{1}{2}\left(\frac{\|\bm{\beta}_{g}\|_{2}^{2}}{\sigma^{2}}\frac{1}{\gamma_{g}^{2}}+\lambda_{2}^{2}\gamma_{g}^{2}\right)\right\} d\gamma_{g}^{2}\\
 & \propto\exp\left\{ -\frac{\lambda_{1}}{\sigma}\|\bm{\beta}_{g}\|_{1}-\frac{\lambda_{2}}{\sigma}\|\bm{\beta}_{g}\|_{2}\right\} .
\end{align*}

\section{Gibbs Sampler for BSGL}

The joint posterior probability density function of $\bm{\beta},\bm{\tau},\bm{\gamma},\sigma^{2}$
given $\bm{Y},\bm{X}$ is
\begin{align*}
& \pi\left(\bm{\beta},\bm{\tau},\bm{\gamma},\sigma^{2}|\bm{Y},\bm{X}\right)\\
 &\quad{} \propto\left(\sigma^{2}\right)^{-n/2}\exp\left\{ -\frac{1}{2\sigma^{2}}\left(\bm{Y}-\bm{X}\bm{\beta}\right)^{T}\left(\bm{Y}-\bm{X}\bm{\beta}\right)\right\} \\
 &\quad{} \times\prod_{g=1}^{G}\left[\left(\sigma^{2}\right)^{-\frac{m_{g}}{2}}\prod_{j=1}^{m_{g}}\left(\frac{1}{\tau_{gj}^{2}}+\frac{1}{\gamma_{g}^{2}}\right)^{\frac{1}{2}}\exp\left\{ -\frac{1}{2\sigma^{2}}\sum_{j=1}^{m_{g}}\left(\frac{1}{\gamma_{g}^{2}}+\frac{1}{\tau_{gj}^{2}}\right)\beta_{gj}^{2}\right\} \right]\\
 &\quad{} \times\prod_{g=1}^{G}\left[\left(\prod_{j=1}^{m_{g}}\left(\frac{1}{\tau_{gj}^{2}}+\frac{1}{\gamma_{g}^{2}}\right)^{-\frac{1}{2}}(\tau_{gj}^{2})^{-\frac{1}{2}}\right)(\gamma_{g}^{2})^{-\frac{1}{2}}\exp\left\{ -\frac{\lambda_{1}^{2}}{2}\sum_{j=1}^{m_{g}}\tau_{gj}^{2}-\frac{\lambda_{2}^{2}}{2}\gamma_{g}^{2}\right\} \right]\frac{1}{\sigma^{2}}\\
 &\quad{} \propto(\sigma^{2})^{-\frac{n+m}{2}-1}\prod_{g=1}^{G}\left[\prod_{j=1}^{m_{g}}(\tau_{gj})^{-\frac{1}{2}}(\gamma_{g}^{2})^{-\frac{1}{2}}\right]\exp\left\{ -\frac{\lambda_{1}^{2}}{2}\sum_{g=1}^{G}\sum_{j=1}^{m_{g}}\tau_{gj}^{2}-\frac{\lambda_{2}^{2}}{2}\sum_{g=1}^{G}\gamma_{g}^{2}\right\} \\
 &\quad{} \times\exp\left\{ -\frac{1}{2\sigma^{2}}\left(\bm{Y}-\bm{X}\bm{\beta}\right)^{T}\left(\bm{Y}-\bm{X}\bm{\beta}\right)-\frac{1}{2\sigma^{2}}\sum_{g=1}^{G}\sum_{j=1}^{m_{g}}\left(\frac{1}{\gamma_{g}^{2}}+\frac{1}{\tau_{gj}^{2}}\right)\beta_{gj}^{2}\right\}.
\end{align*}
Then we can generate from the posterior distribution using the following
full conditional posteriors,
\begin{align*}
\sigma^{2}|\text{rest} & \sim\text{Inverse Gamma}\left(\frac{m+n}{2},\frac{1}{2}\left(\bm{Y}-\bm{X}\bm{\beta}\right)^{T}\left(\bm{Y}-\bm{X}\bm{\beta}\right)+\frac{1}{2}\bm{\beta}^{T}\bm{V}^{-1}\bm{\beta}\right),\\
\gamma_{g}^{2}|\text{rest} & \stackrel{ind}{\sim}\text{Inverse Gaussian}\left(\frac{\sigma\lambda_{2}}{\|\bm{\beta}_{g}\|_{2}^{2}},\lambda_{2}^{2}\right),\quad g=1,\dots,G,\\
\tau_{gj}^{2}|\text{rest} & \stackrel{ind}{\sim}\text{Inverse Gaussian}\left(\frac{\sigma\lambda_{1}}{|\beta_{gj}|},\lambda_{1}^{2}\right),\quad g=1,\dots,G;\  j=1,\dots,m_{g},\\
\bm{\beta}|\text{rest} & \sim N\left((\bm{X}^{T}\bm{X}+\bm{V}^{-1})^{-1}\bm{X}^{T}\bm{Y},(\bm{X}^{T}\bm{X}+\bm{V}^{-1})^{-1}\right),\\
\lambda_{1}^{2}|\text{rest} & \sim\text{Gamma}\left(p+1,\frac{\|\bm{\tau}\|_{2}^{2}}{2}+d_{1}\right),\\
\lambda_{2}^{2}|\text{rest} & \sim\text{Gamma}\left(\frac{G}{2}+1,\frac{\|\bm{\gamma}\|_{2}^{2}}{2}+d_{2}\right)
\end{align*}
where
\[
\bm{V}=\left(\begin{array}{cccc}
\bm{V}_{1} & 0 & \dots & 0\\
0 & \bm{V}_{2} & \dots & 0\\
\vdots & \vdots & \ddots & \vdots\\
0 & 0 & \dots & \bm{V}_{G}
\end{array}\right).
\]

\bibliographystyle{ba}

\begin{acknowledgement}
Ghosh's research was partially supported by NSF Grants DMS-1007417
and SES-1026165. The authors want to thank the reviewers for their
helpful comments which substantially improved the paper.
\end{acknowledgement}

\end{document}